\documentclass[12pt,a4paper]{amsart}

\usepackage{hyperref}
\usepackage{paralist, mathdots} %

\usepackage{amssymb} %
\usepackage{amsmath,tikz,wasysym,multirow} %
\usepackage{amsthm,makecell}
\newtheorem{proposition}{Proposition}[section]
\newtheorem{conjecture}[proposition]{Conjecture}
\newtheorem{theorem}[proposition]{Theorem}

\newtheorem{corollary}[proposition]{Corollary}
\newtheorem{lemma}[proposition]{Lemma}
\theoremstyle{definition}
\newtheorem{remark}[proposition]{Remark}
\newtheorem{definition}[proposition]{Definition}

\newcommand{\npmatrix}[1]{\left( \begin{matrix} #1 \end{matrix} \right)}
\newcommand{\R}{\mathbb{R}}
\let\C\relax
\newcommand{\C}{\mathbb{C}}

\newcommand{\mc}{\mathcal}
\DeclareMathOperator{\mr}{mr}
\DeclareMathOperator{\rk}{rank}
\DeclareMathOperator{\tr}{tr}  
  
\DeclareMathOperator{\rank}{rank}

\DeclareMathOperator{\diam}{diam}

\newcommand{\qtext}[1]{\quad\text{#1}\quad}
\newcommand{\qand}{\qtext{and}}
    \definecolor{helena}{rgb}{.2,.8,.4}
    \definecolor{polona}{rgb}{.8,.2,.2}
    \definecolor{rupert}{rgb}{0,.5,.5}
   \definecolor{todo}{rgb}{.2,.2,.8}

\DeclareMathOperator{\dup}{dup}
\DeclareMathOperator{\jdup}{jdup}

\begin{document}
\title[A Nordhaus-Gaddum conjecture for $q$]
{A Nordhaus-Gaddum conjecture\\ for the minimum number of distinct eigenvalues of a graph}
\author{Rupert H. Levene, Polona Oblak, Helena \v Smigoc}
\address[R.~H.~Levene and H.~\v Smigoc]{School of Mathematics and Statistics, University College Dublin, Belfield, Dublin 4, Ireland}
\email{rupert.levene@ucd.ie}\email{helena.smigoc@ucd.ie}
\address[P.~Oblak]{Faculty of Computer and Information Science, University of Ljubljana, Ve\v cna pot 113, SI-1000 Ljubljana, Slovenia}
\email{polona.oblak@fri.uni-lj.si}
\bigskip

\maketitle 

\begin{abstract}
  We propose a Nordhaus-Gaddum conjecture for $q(G)$, the
  minimum number of distinct eigenvalues of a symmetric matrix
  corresponding to a graph $G$: for every graph $G$ excluding four exceptions, we conjecture that
  $q(G)+q(G^c)\le |G|+2$, where $G^c$ is the complement of~$G$. We
  compute $q(G^c)$ for all trees and all graphs~$G$ with $q(G)=|G|-1$,
  and hence we verify the conjecture for trees, unicyclic graphs,
  graphs with $q(G) \leq 4$, and for graphs with $|G|\le7$.
\end{abstract}

\section{Introduction}

Let $G$ be a graph (by which we will always mean a finite, undirected, simple graph) with vertex set $V(G)=\{1,\dots,n\}$ and edge set~$E(G)$, and consider
$S(G)$, the set of all real symmetric $n \times n$ matrices
$A=(a_{ij})$ such that, for $i \neq j$, $a_{ij} \neq 0$ if and only if
$\{i,j\} \in E(G)$, with no restriction on the diagonal entries of
$A$.  The Inverse Eigenvalue Problem for Graphs (IEPG) is the problem
of characterising all lists of eigenvalues of matrices in
$S(G)$ for any given graph $G$. The IEPG motivates the study of
several parameters, for example, the widely studied minimum rank of a
graph:
 $$\mr (G)=\min \{\rk(A)\colon A \in S(G)\}$$
 and the minimum number of distinct eigenvalues of a graph:
$$q(G) = \min\{q(A)\colon A \in S(G)\}$$
where $q(A)$ denotes the number of distinct eigenvalues of a square matrix~$A$. The parameter $q(G)$ is the focus of a growing body of literature, e.g.,~\cite{ MR3118943, MR3665573, 2017arXiv170801821B,MR2735867,MR1899084}, and it is quickly becoming an important parameter in Spectral Graph Theory. 

Minimum rank has been extensively studied; for an overview, we refer
the reader to two surveys~%
\cite{MR2350678,2011arXiv1102.5142F}. One of the most prominent open questions
associated with $\mr (G)$ is the so-called graph complement conjecture
for minimum rank that arose from the American Institute of Mathematics
workshop \emph{Spectra of Families of Matrices described by Graphs,
  Digraphs, and Sign Patterns}~\cite{aims}. To state the conjecture,
recall that the \emph{complement} $G^c$ of a graph $G$ is the graph
with the same vertex set as~$G$, such that two vertices are adjacent
in $G^c$ if and only if they are not adjacent in $G$.
\begin{conjecture}[The graph complement conjecture for minimum rank]
For any graph $G$ we have
\begin{equation*}\label{mr-conjecture}
 \mr(G)+\mr(G^c)\leq |G|+2.
 \end{equation*}
\end{conjecture}
It has been proven that this conjecture is satisfied by several families of graphs~\cite{MR2917415,MR3194944}, but it remains open in general. Note that equality in the conjecture is achieved, for example, for paths. The conjecture is an example of a Nordhaus-Gaddum type problem. Nordhaus and Gaddum~\cite{MR0078685} bounded the sum and the product
of the chromatic number of a graph and its complement, in terms of the order of the graph, and since then relations of this type for other graph invariants have come to be associated with their names~\cite{MR3526413,MR3015299,MR1826471}. 

It is immediate~\cite[Proposition~2.5]{MR3118943} that 
\begin{equation}\label{eq:mrbound}
  q(G)\le \mr(G)+1.
\end{equation}
As a result, for those graphs~$G$ satisfying the graph complement conjecture for minimum rank, we have
\[ q(G)+q(G^c)\leq |G|+4.\]
We conjecture that with only four exceptions, this Nordhaus-Gaddum type bound for $q(G)$ can be improved, as follows.

\begin{conjecture}\label{conjecture}
For any graph~$G$, the inequality
\begin{equation}\label{eq:conjecture}
 q(G)+q(G^c)\leq |G|+2
 \end{equation}
holds if and only if $G$ is not one of the four graphs $P_4$, $P_5$, $P_5^c$, and $T_{2,2}$ (see Figure~\ref{fig:exceptions}).
\end{conjecture}

\begin{figure}[hbt]\label{fig:exceptions}
\centering
\begin{tabular}{ccccccc}	
 \begin{tikzpicture}[style=thick, scale=0.7]
		\draw \foreach \x in {1,...,3} {
				(\x,0) -- (\x + 1,0) };
		\draw[fill=white] \foreach \x in {1,...,4} {
		                    (\x,0) circle (1mm)
		};		
		\end{tikzpicture}&\qquad&
\begin{tikzpicture}[style=thick, scale=0.7]
		\draw \foreach \x in {1,...,4} {
				(\x,0) -- (\x + 1,0) };
		\draw[fill=white] \foreach \x in {1,...,5} {
		                    (\x,0) circle (1mm)
		};		
		\end{tikzpicture}&\qquad&
\begin{tikzpicture}[style=thick, scale=0.7]
		\draw (1,1)--(1,0)--(2,0)--(2,1)--(1,1);
                \draw (2,0)--(3,.5)--(2,1);
		\draw[fill=white] (1,0)  circle (1mm)(2,0)  circle (1mm) (2,1)  circle (1mm) (1,1)  circle (1mm) (3,.5) circle (1mm) (2,1) circle (1mm);
		\end{tikzpicture}&\qquad&
\begin{tikzpicture}[style=thick, scale=0.7]
		\draw \foreach \x in {2,3,4} {
				(\x,0) -- (\x + 1,0) };
		\draw (3,0)--(3.5,1)--(4,0);
		\draw[fill=white] \foreach \x in {2,3,...,5} {
		                    (\x,0) circle (1mm)
		};		
		\draw[fill=white] (3.5,1) circle (1mm);
		\end{tikzpicture}\\
$P_4$&&$P_5$&&$P_5^c$&&$T_{2,2}$\\
\end{tabular}
\caption{Four  graphs which do not satisfy inequality \eqref{eq:conjecture}. In all four cases, we have $q(G)+q(G^c)=8$.}
\end{figure}
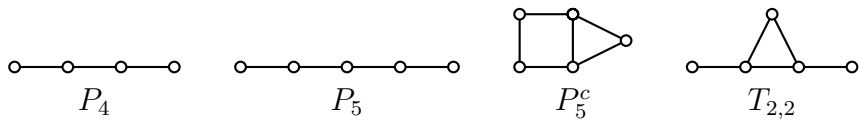

As we will see in Section~\ref{sec:trees}, inequality~\eqref{eq:conjecture} is saturated for all paths on six or more vertices, so $2$ cannot be replaced by a smaller constant on the right hand side.

Motivated by Conjecture~\ref{conjecture}, the main aim of this paper is to determine $q(G^c)$ for some families of graphs. After recalling some notation and a selection of preliminary results in the remainder of this introduction, we establish some constructive tools that are useful in determining $q(G)$ in Section~\ref{contructions}. In Section~\ref{sec:trees} we compute $q(G^c)$ for trees, 
and in Section~\ref{q=n-1}  we compute $q(G^c)$ for graphs $G$ with $q(G) \geq |G|-1$.
In Section~\ref{sec:conjecture} we bring together the results of previous sections to examine Conjecture~\ref{conjecture}. In particular, we prove that the conjecture holds for all graphs
$G$ with $q(G) \geq |G|-1$ (Theorem~\ref{thm:n-1}); for bipartite graphs, for graphs with small minimum rank and for certain joins and unions (Corollary~\ref{cor:bipartite}); for unicyclic graphs (Theorem~\ref{unicyclic}); and for all graphs with at most 7 vertices (Theorem~\ref{lemma:67}).%

\subsection{Notation}

Below we collect some standard notation and terminology for graphs and matrices;  some of it has already been used above. For a graph $G=(V(G),E(G))$, the \emph{order} of~$G$ is 
$|G|=|V(G)|$. We will routinely identify isomorphic graphs without further comment.
We say that two distinct vertices $x,y\in V(G)$ are \emph{adjacent} in~$G$ if $\{x,y\}\in E(G)$.  A sequence of~$k$ edges $\{x_0,x_1\}, \{x_1,x_2\},\ldots,\{x_{k-1},x_{k}\}$ in a graph is called 
\emph{a path of length $k$} between the vertices~$x_0$ and $x_k$. The \emph{distance} between two vertices is the length of the shortest path between them and the \emph{diameter} $\diam(G)$ is the largest distance between any two vertices of the graph. If only one shortest path exists between two vertices, it is called a \emph{unique shortest path} between them.

The \emph{open neighbourhood} of a vertex $v$ in graph $G$ is denoted
by %
$N(v)$, and consists of all vertices at distance 1 from vertex $v$,
i.e., all vertices adjacent to $v$ in $G$. The \emph{degree} of~$v$ in~$G$ is $|N(v)|$, and a \emph{leaf} of~$G$ is a vertex of degree~$1$.

The \emph{complement}  $G^c$ of a graph $G$ is the graph with vertex set $V(G)$ such that two vertices are adjacent in $G^c$ if and only if they are not adjacent in $G$.  The \emph{join} $G \vee H$ of $G$ and $H$ is the disjoint graph union $G \cup H$ together with all the possible edges joining the vertices in $G$ to the vertices in $H$.
 The \emph{complete graph} on $n$ vertices  will be denoted by $K_n$ and the \emph{complete bipartite graph} on two disjoint sets of cardinality $m$ and $n$ by $K_{m,n}$. We write $P_n$ for the \emph{path} on~$n$ vertices, and note in passing that for any graph~$G$, we have $q(G)=|G|$ if and only if $G=P_{|G|}$.
 The \emph{chromatic number}~$\chi(G)$ of a graph~$G$ is the smallest number of colours needed to colour its vertices in such way that no two adjacent vertices have the same colour; this is called a $k$-colouring.  We say that $G$ is \emph{$2$-colourable} if its chromatic number is equal to 2, or equivalently if $G$ is bipartite, and in this case we take the colours to be black and white.
Enumerated graphs (such as G188) follow the numbering scheme of~\cite{Read:2005:AG:1051300}.

An \emph{induced subgraph} of~$G$ is a graph $H=(V(H),E(H))$ where
$V(H)\subseteq V(G)$ and $E(H)$ contains every edge in~$E(G)$ whose
endpoints are both in $V(H)$. On the other hand, a graph
$H=(V(H),E(H))$ is simply a \emph{subgraph} of $G$ (written
$H\subseteq G$ or $G\supseteq H$) if $V(H)\subseteq V(G)$ and
$E(H)\subseteq E(G)$. Most subgraphs we will consider will be of the
latter type. A subgraph~$H$ is a \emph{spanning subgraph} of $G$ if
$V(H)=V(G)$. We write $G\setminus\{e\}$ for the spanning subgraph of $G$ consisting of $G$ with one edge~$e$ removed. Similarly, if $v\in V(G)$, then $G\setminus\{v\}$
denotes the induced subgraph of $G$ with vertex set
$V(H)=V(G)\setminus\{v\}$.

By $I_n$ we denote the $n \times n$ identity matrix and by $0_{m,n}$
we denote the $m \times n$ zero matrix. We omit the indices if the
size of the matrix is clear from the context. A matrix is said to be
\emph{nonnegative} if each of its entries is a nonnegative real
number. For any matrix~$B$, we write $\|B\|$ for the spectral norm
of~$B$, that is, the largest singular value of~$B$. We
write~$\sigma(A)$ for the multiset of eigenvalues of a square
matrix~$A$. The \emph{multiplicity list} of $A$ is the multiset of
cardinality $q(A)$ containing the multiplicity in~$\sigma(A)$ of each
of the distinct eigenvalues of~$A$. We will use $\circ$ to denote the
Hadamard product of matrices, and $[A,X]$ to denote the commutator $AX-XA$
of two $n\times n$ matrices $A$ and $X$. The \emph{pattern} of a
matrix~$A=(a_{ij})$ is the set of pairs~$(i,j)$ so that
$a_{ij}\ne0$. Finally, recall that a square matrix~$A$ is
\emph{reducible} if it may be put into block $2\times 2$
upper-triangular form (in the symmetric case, a block $2\times 2$
direct sum), where neither block is trivial, by conjugating~$A$ with a
permutation matrix; otherwise,~$A$ is said to be \emph{irreducible}.

\subsection{Preliminaries}

Powerful tools to advance the IEPG in general, and the minimum distinct eigenvalue problem in particular, were introduced in~\cite{MR3665573}. We follow~\cite{MR3665573}, and say that an $n\times n$ symmetric matrix $A$ has 
the \emph{Strong Spectral Property} (SSP) if 
 the zero matrix is the only symmetric matrix $X$ for which
 $$A\circ X=I \circ X=0 \, \text{ and }\, [A,X]=0,$$
 and we say that $A$ satisfies the Strong Multiplicity Property (SMP) if the zero matrix is the only symmetric matrix $X$ with the properties
 $$A\circ X=I \circ X=0, \, [A,X]=0 \, \text{ and }\, \tr(A^iX)=0 \text{ for } i=0,\ldots,n-1.$$
 The SSP and the SMP are generalisations of the Strong Arnold Property
 (SAP) introduced in~\cite{MR1673503}.  The SMP is stronger than the
 SAP and weaker than the SSP; for matrices with two distinct
 eigenvalues, the SMP and SSP coincide.

Below we list some results from~\cite{MR3665573} %
on the SMP and the SSP that illustrate the importance of those properties in the IEPG, and that we will need later in this work.

\begin{theorem} (\cite[Theorem~34]{MR3665573})\label{thm:oplus}
  For $i = 1, 2$, let $A_i$ be a symmetric real $n_i\times n_i$ matrix. Then $A = A_1 \oplus A_2$ has the SSP (respectively, SMP) if and only if both $A_1$ and $A_2$ have the SSP (respectively, SMP) and $\sigma(A_1) \cap \sigma(A_2) = \emptyset$. 
\end{theorem}

The following is a special case %
of~\cite[Theorem~36]{MR3665573}:

\begin{theorem}%
\label{thm:smp}
Let $G$ be a graph and let $\tilde G$ be a spanning subgraph of $G$. If $\tilde A \in S(\tilde G)$ has the SSP (respectively, SMP), then there exists $A \in S(G)$ with the SSP (respectively, SMP) so that $\sigma(A)=\sigma(\tilde A)$.
\end{theorem}

Motivated by the theorems above, the following parameters were introduced in~\cite{MR3665573}:
\begin{align*}
 q_M(G) &= \min\{q(A)\colon A \in S(G) \text{ and } A \text { has  the SMP}\}, \\
q_S(G) &= \min\{q(A)\colon A \in S(G) \text{ and } A \text { has  the SSP}\}.
\end{align*}
 Note that $$q(G) \leq q_M(G) \leq q_S(G).$$

 A lower bound on $q$, that is achieved for several families of graphs, was given in~\cite[Theorem~3.1]{MR2735867} and~\cite[Theorem~3.2]{MR3118943}. There it was shown that 
 \begin{equation}\label{eq:Fallat}
 q(G) \geq d(G)+1,
 \end{equation}
 where $d(G)$ is the number of edges in the longest unique shortest path between any two vertices in $G$ (if $G$ has no edges, we define $d(G)=0$).

We will also use the following upper bound, which appears in~\cite[Corollary~49]{MR3665573}. 
Let $c(G)$ denote the circumference of $G$, i.e., the number of vertices of the largest cycle which is a subgraph of $G$. Then
\begin{equation}\label{eq:cyclebound}
q(G) \leq |G|-\left\lfloor \frac{c(G)}{2}\right\rfloor.
\end{equation}
 
\bigskip

\section{Constructions}\label{contructions}
\newcommand{\M}[2][B]{M(#1,#2)}
\newcommand{\Mhat}[3][B]{\widehat{M}(#1,#3,#2)}

Recall that a symmetric matrix~$M$ is orthogonal if $M^2=I$, or equivalently, if each of its eigenvalues is either $-1$ or $1$. We start this section with two related constructions that each give a family of orthogonal matrices whose patterns we will often be able to control.

Let $B$ be an $m \times n$ matrix with $\|B\|<1$. Choose $\alpha \in [-1,1]$ and define
\begin{equation}\label{eq:M}
  \M{\alpha}=\npmatrix{\sqrt{I_m-\alpha^2 BB^T} & \alpha B \\ \alpha B^T & -\sqrt{I_n-\alpha^2 B^TB}}.
 \end{equation}
 Note that $\M{\alpha}^2=I_{m+n}$. 
 
 Assuming that $\rank( B)< m$, let $v$ be a unit vector such that $B^Tv=0.$ Then the matrix 
\begin{equation}\label{eq:N}
\Mhat{\alpha}{v}=\begin{pmatrix}0&v^T&0_{1,n}\\v & \sqrt{I_{m}-(\alpha^2 BB^T+vv^T)} &\alpha B\\0_{n,1}& \alpha B^T&-\sqrt{I_n-\alpha^2 B^TB}
    \end{pmatrix}
\end{equation}
  is well defined for $\alpha \in  [-1,1]$, and $\Mhat{\alpha}{v}^2=I_{1+m+n}.$
 
In our applications, we will first fix the matrix $B$ and then demand that the diagonal blocks of $\M{\alpha}$ and $\Mhat{\alpha}{v}$ containing the square roots do not have any zero elements. To make this precise we need the following definition. 

\begin{definition}
Let $A$ be an $n \times n$ matrix. We say that $A$ is \emph{generalised irreducible} if for every pair $(i,j)$, $i,j=1,2,\ldots,n$, there exists a positive integer $k$ such that $(A^k)_{ij} \neq 0$. 
\end{definition}

The simplest examples of  generalised irreducible matrices are matrices with no zero elements, and nonnegative irreducible matrices. 

\begin{lemma}
Let $A$ be a symmetric $n \times n$ generalised irreducible matrix with $\|A\|<1$.
\begin{enumerate}
\item $\sqrt{I_n-\alpha A}$ has no zero entries for all but a finite number of $\alpha$ in the interval $ [-1,1]$. 
\item If $v$ is a unit vector with  $Av=0$, then $$\sqrt{I_{n}-(\alpha A+vv^T)}=\sqrt{I_{n}-\alpha A}-vv^T$$ has no zero entries for all but a finite number of $ \alpha \in [-1,1]$. 
\end{enumerate}
\end{lemma}

\begin{proof}
The $(i,j)$--th element of $\sqrt{I_n-\alpha A}$, and the $(i,j)$--th element of $\sqrt{I_n-\alpha A}-vv^T$ are both, for some %
$\gamma_{ij}$ independent of $\alpha$, of the following form:
\[ f_{ij}(\alpha):=\gamma_{ij} -\sum_{j=1}^{\infty}  c_k a_{ij}(k) \alpha^k,\]
where $c_k>0$ is the $k$th coefficient of the Taylor series at $x=0$ for the function $x\mapsto -\sqrt{1-x}$, and $a_{ij}(k)$ denotes the $(i,j)$-th element of $A^k$. It is easy to check that the series for $f_{ij}(\alpha)$ converges absolutely for $\alpha$ in the open interval $U=(-\|A\|^{-1},\|A\|^{-1})$ so the function $f_{ij}$ is real-analytic on~$U$. Since $A$ is generalised irreducible, there exists a $k >0$ so that $a_{ij}(k) \neq 0$. Hence $f_{ij}$ is not a constant, so $f_{ij}$ has finitely many zeroes on the compact set $[-1,1]\subseteq U$. The union over $1\leq i,j\leq n$ of these zero sets is also finite, and the statement follows.
\end{proof}

Note that in the second item we do not need $A$ to be generalised irreducible to reach the conclusion, if we assume instead that  $v$ has no zero entries.  We summarise our observations so far in the corollary below. 

\begin{corollary}\label{cor:sqrt}
  Let $B$ be an $m \times n$ matrix with $\|B\|<1$ for which $BB^T$
  and $B^TB$ are both generalised irreducible matrices.
  \begin{enumerate}
  \item For any $\alpha\in [-1,1]$, the matrix $\M{\alpha}$
    in~\eqref{eq:M} is well defined and has at most two distinct
    eigenvalues. Moreover, for all but finitely many
    $\alpha\in[-1,1]$, the diagonal blocks containing the square roots
    in $\M{\alpha}$ have no zero entries.
  \item If %
    $v$ is a unit vector in $\R^m$ with
    $B^Tv=0$, then the same holds for the matrices $\Mhat{\alpha}{v}$ in~\eqref{eq:N}.
  \end{enumerate}
\end{corollary}

A nonnegative matrix is generalised irreducible precisely when it is irreducible, and by definition, the irreducibility of any 
matrix depends only on its pattern. In the next lemma we connect the irreducibility of $BB^T$ with the pattern of $B$.

 \begin{lemma}\label{lem:reducible}
Let $B$ be an $m \times n$ nonnegative matrix. Then $BB^T$ is reducible if and only if there exists an $m \times m$ permutation matrix $P$ and an $n\times n$ permutation matrix $Q$ such that $PBQ$ is of the form \begin{equation*}PBQ=\npmatrix{B' & 0_{m_1,n_2} \\ 0_{m_2,n_1} & B''},
\end{equation*}
where $m_1 \neq 0$, $m_2 \neq 0$ and $n_1,n_2\ge0$ with $n_1+n_2=n$ and $m_1+m_2=m$.
\end{lemma}

\begin{proof}
First we note that, by the definition of reducibility, the symmetric matrix $BB^T$ is reducible if and only if 
\begin{equation}\label{eq:reducible}
PBB^TP^T=\npmatrix{C_1 & 0_{m_1,m_2} \\ 0_{m_2,m_1} & C_2}
\end{equation}
 for some permutation matrix $P$ and some $m_1,m_2\ne 0$.  Let us write 
  $$PB=\npmatrix{B_1 \\ B_2},$$
  where $B_1$ is an $m_1 \times n$ matrix, and $B_2$ is an $m_2 \times n$ matrix. From (\ref{eq:reducible}) we deduce that $B_1B_2^T=0$. %
  Let $n_2\in \{0,1,\dots,n\}$ be the number of zero columns of $B_1$, let $n_1=n-n_2$ and let $Q$ be a permutation matrix such that 
   $$B_1Q=\npmatrix{B' & 0_{m_1,n_2}},$$
   where the $m_1\times n_1$ matrix $B'$ has no zero columns. From nonnegativity and $B_1Q(B_2Q)^T=0$, we conclude that 
 $$B_2Q=\npmatrix{0_{m_2, n_1} & B''},$$
thus finishing the proof. 
\end{proof}

In~\cite[Theorem~35]{MR3665573}
 it was shown that $q(G^c)\le 2\chi(G)$  for a graph~$G$. In particular, $q(G^c)\le 4$ for any bipartite graph~$G$. As we will now see, many bipartite graphs have $q(G^c)=2$.

 \begin{theorem}\label{thm:bipartite}
 Let $G$ be a $2$-colourable graph which admits a $2$-colouring with $m$ black vertices and $n$ white vertices, where $|G|=m+n \geq 3$. 
 If $G$  does not contain \[K_{m_1, n_1} \cup K_{m_2, n_2}\] 
 as a spanning subgraph, where $m=m_1+m_2$, $n=n_1+n_2$, $m_1,m_2,n_1,n_2\ge0$ and either $m_1m_2\ne0$ or $n_1n_2\ne0$, %
 then $q(G^c)=2$. %
 \end{theorem}
 
\begin{proof}
Since $|G|\geq 3$ and $G$ is bipartite, $G^c$ contains at least one edge and so $q(G^c)\ge2$.
Let $1,2,\ldots,m$ be the black vertices, and $m+1, \ldots,m+n$ be the white vertices in the $2$-colouring of $G$. %
Let~$B=[b_{i,j}]$ be any nonnegative $m\times n$ matrix with $\|B\|<1$, satisfying
\begin{equation*}%
b_{i,j}= 0\iff \{i,m+j\}\text{ is an edge of~$G$.}
\end{equation*} We claim that $BB^T$ and $B^TB$ are irreducible. Indeed, by Lemma~\ref{lem:reducible}, if $BB^T$ is reducible then we have permutation matrices $P,Q$ so that $PBQ$ is of the form
 \begin{equation*}
   PBQ=\npmatrix{B' & 0_{m_1,n_2} \\ 0_{m_2,n_1} & B''}
\end{equation*}
where $m_1 \neq 0$, $m_2 \neq 0$ and $n_1,n_2\ge0$. This implies that if $C=[c_{i,j}]$ is any $m\times n$ matrix so that $b_{i,j}=0\iff c_{i,j}\ne 0$, then
\[ \npmatrix{0&C\\C^T&0}\in S(G)
\qand PCQ=\npmatrix{C'&C_{12}\\C_{21}&C''}\]
where $C_{12}$ and $C_{21}$ have no zero entries. 
Then
\[\npmatrix{P&0\\0&Q^T}\npmatrix{0&C\\C^T&0}\npmatrix{P^T&0\\0&Q}= \npmatrix{0&PCQ\\(PCQ)^T&0}\]
is in $S(\tilde G)$ where $\tilde G$ is isomorphic to $G$ with the black and white vertices permuted by $P$ and $Q^T$, respectively. Hence $G$ contains $K_{m_1,n_1}\cup K_{m_2,n_2}$, contrary to hypothesis, so $BB^T$ is irreducible. Similary, $B^TB$ is irreducible.

By Corollary~\ref{cor:sqrt}, the matrix
 \begin{equation*}%
 M=\M{1}=\npmatrix{\sqrt{I_m-BB^T} & B \\ B^T & -\sqrt{I_n-B^TB}}
 \end{equation*}
 has $q(M)\leq 2$ and the diagonal blocks of $M$ have all their entries different from zero, so $M\in S(G^c)$. Hence, $q(G^c)\leq q(M)=2$. 
 \end{proof}

We will also make use of the following construction from~\cite{MR2098598}.

\begin{lemma}%
\label{HS04}
Let $B$ be a symmetric $m \times m$ matrix and let $u$ be an eigenvector of~$B$ with eigenvalue $\mu$, normalised so that $u^Tu=1$. If 
 \begin{equation}\label{c1f1}
  A=\npmatrix{A_1 & b \\
               b^T & \mu}
 \end{equation}
 is any symmetric $n \times n$ matrix with a diagonal element $\mu$, then the $(m+n-1)\times (m+n-1)$ matrix
  \[C=\npmatrix{A_1 & bu^T \\
               ub^T & B}\]
               has  $\sigma(C)=\sigma(A)\cup (\sigma(B)\setminus\{\mu\})$. 
\end{lemma}

In particular, Lemma~\ref{HS04} gives us a tool to manage duplicated vertices in a graph. Let $G$ be a graph containing a vertex $v$.  Recall that $N(v)$ %
denotes the open neighbourhood of $v$, i.e., the set of all vertices in $G$ adjacent to $v$. 
We can form a new graph
$\dup(G,v)$, by \emph{duplicating}~$v$; that is, by augmenting~$G$ with a new
vertex $w$ and extra edges joining $w$ to every vertex in $N(v)$, so that in $\dup(G,v)$ we have $N(v)=N(w)$ and $v$ and $w$ are not neighbours. 

Alternatively, we can
form the ``joined duplicated vertex'' graph $\jdup(G,v)$ which is equal to $\dup(G,v)$
with an extra edge joining $v$ and its duplicate vertex, $w$.

 \begin{center}
 \begin{tabular}{ccc}
	\begin{tikzpicture}[style=thick, scale=0.7]
		\draw \foreach \x in {-1,0,1,2,3} {
				(\x,0) -- (\x + 1,0) };
		\draw (3,0)--(3.5,1)--(4,0);
		\draw[fill=white] \foreach \x in {-1,0,2,3,4} {
		                    (\x,0) circle (1mm)
		};		
		\draw[fill=white] (3.5,1) circle (1mm);
		\draw[fill=gray] (1,0) circle (1mm);
		\node [above] at (1,0) {$v$};
		\phantom{\draw[fill=white] (3.5,-1) circle (1mm);}
		\end{tikzpicture}
		&
			\begin{tikzpicture}[style=thick, scale=0.7]
		\draw \foreach \x in {-1,2,3} {
				(\x,0) -- (\x + 1,0) };
		\draw (3,0)--(3.5,1)--(4,0);
		\draw (2,0)--(1,1)--(0,0)--(1,-1)--(2,0);
		\draw[fill=white] \foreach \x in {-1,0,2,3,4} {
		                    (\x,0) circle (1mm)
		};		
		\draw[fill=white] (3.5,1) circle (1mm);
		\draw[fill=gray] (1,1) circle (1mm) (1,-1) circle (1mm);
		\end{tikzpicture}
		&
			\begin{tikzpicture}[style=thick, scale=0.7]
		\draw \foreach \x in {-1,2,3} {
				(\x,0) -- (\x + 1,0) };
		\draw (3,0)--(3.5,1)--(4,0);
		\draw (2,0)--(1,1)--(0,0)--(1,-1)--(2,0);
		\draw (1,1)--(1,-1);
		\draw[fill=white] \foreach \x in {-1,0,2,3,4} {
		                    (\x,0) circle (1mm)
		};		
		\draw[fill=white] (3.5,1) circle (1mm);
		\draw[fill=gray] (1,1) circle (1mm) (1,-1) circle (1mm);
		\end{tikzpicture}
\\
$G$ & $\dup(G,v)$ & $\jdup(G,v)$  \\
	\end{tabular}
\end{center}
Observe that 
\begin{equation}\label{jdup}
 \dup(G,v)^c = \jdup(G^c,v).
 \end{equation}
First we bound $q(\jdup(G,v))$. 

\begin{lemma}\label{prop:dup}
If $G$ is a non-empty graph, then $$q(\jdup(G,v))\leq q(G).$$ 
\end{lemma}
 \begin{proof}
  The statement is contained
  in~\cite[Theorem~3]{2017arxiv170802438} (alternatively it is an easy
  consequence of Lemma~\ref{HS04}). 
   \end{proof}

It is clear from the lemma above that, if $G$ is not a complete graph, then $q(\dup(G,v)^c)\leq q(G^c)$ by relation~\eqref{jdup}.

Bounding $q(\dup(G,v))$ with $q(G)$ is not as straightforward, but it can be done, if a matrix realising $q(G)$ has one of its eigenvalues equal to one of its diagonal elements. More precisely, the following lemma is contained in~\cite[Theorem~3]{2017arxiv170802438}; we include a brief proof.

\begin{lemma}\label{lem:dup}
  If $v$ is a %
  vertex of $G$ and $A\in S(G)$ with $\lambda=a_{v,v}\in \sigma(A)$,
  then there is a matrix $C\in S(\dup(G,v))$ with
  $\lambda=c_{v,v}$ %
  so that the spectra of $A$ and $C$ are equal as sets,
  with the multiplicity of $\lambda$ in $\sigma(C)$ increased by $1$.
\end{lemma}
\begin{proof}
  Apply Lemma~\ref{HS04} with $B=\lambda I_2$ and $u$ any normalised vector without zero entries.
\end{proof}

\section{Complements of trees}\label{sec:trees}

In this section we will compute $q(G)$ for complements of trees. After we apply Theorem~\ref{thm:bipartite}, we are left with three exceptional families to consider. 

\begin{proposition}\label{qc=2}
Let $T$ be a tree which admits a $2$-colouring with $m$ black vertices and $n$ white vertices, where $|T|=m+n \geq 3$. 
 If $T$ contains $$K_{m_1, n_1} \cup K_{m_2, n_2},$$ 
where $m=m_1+m_2$, $n=n_1+n_2$, $m_1,m_2,n_1,n_2\ge0$ and either $m_1m_2\ne0$ or $n_1n_2\ne0$, as a spanning subgraph , then $T$ has one of the following forms: 
 \begin{enumerate}
 \item[(type 0):]  $T=K_{1,|T|-1}$
 \item[(type 1):] $T$ is equal to $K_{1,s} \cup K_{1,t}$ with one additional edge
   added in such a way to obtain a tree; or
 \item[(type 2):]  $T\ne K_{1,|T|-1}$, and $T$ is equal to $K_{1,s}$
   with~$t$ vertices appended as leaves,
\end{enumerate}
for some $s,t\ge1$.
 \end{proposition}

\begin{proof}
Note that if $T\ne K_{1,|T|-1}$, then $m,n\ge2$. 
Furthermore, the subgraphs $K_{m_i,n_i}$ need to be acyclic, so $k_1:=\min\{m_1,n_1\}\leq 1$ and $k_2:=\min\{m_2,n_2\}\leq1$. By symmetry, we may assume that $k_1\geq k_2$.   One of our partitions is proper, so $(k_1,k_2)\ne (0,0)$. If $k_1=k_2=1$, then $T$ is of type~1. On the other hand, if $k_1=1$ and $k_2=0$, then $K_{1,n_1}\cup K_{m_2,0}$ is an induced subgraph of $T$, and unless $T= K_{1,|T|-1}$, $T$ is of type~2.
\end{proof}

For type $0$ trees $T$, we have $T^c=K_{|T|-1}\cup \{v\}$, so $q(T^c)=2$. We compute $q(T^c)$ for trees of types 1 and 2 in subsections below. Theorem~\ref{thm:bipartite} and Proposition~\ref{qc=2}, together with the results in the rest of this section give us the following complete description. 

\begin{theorem}\label{thm:treeC}
  If $T$ is a tree, then
$$q(T^c)=
\begin{cases}
4, &\text{for } T=P_4, \smallskip\\ 
3, &\parbox[t]{.8\textwidth}{for $T\neq P_4$  and either $T=S_{m,1}^k$, $k\in \{2,3\}$, $m\ge1$\\
   or $T=W(2,l,(1,d))$, $l\ge0$, $d\ge1,$\smallskip} \\
2, &\text{otherwise.}
\end{cases}$$
\end{theorem}

The graphs $S_{m,n}^k$ and $W(k,l,\delta)$ are defined later in this section. %

\subsection{Type 1 exceptional trees}

 Let us consider the trees of type~1 from Corollary~\ref{qc=2}. 
 Depending on the way we add the final edge to $K_{1,m} \cup K_{1,n}$, three cases can occur. Representative examples are given below: 
 \begin{center}
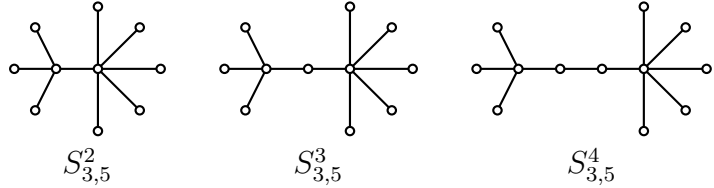
 \begin{tabular}{ccccc}
	\begin{tikzpicture}[style=thick, scale=0.55]
		\draw (0,1.5) -- (0,0) -- (1,1);
		\draw  (1.5,0) -- (0,0);
		\draw  (1,-1) -- (0,0);
		\draw  (0,-1.5) -- (0,0) -- (-1,0) -- (-2,0);				
		\draw  (-1.5,1) -- (-1,0) -- (-1.5,-1);
		\draw[fill=white] (0,1.5) circle (1mm)  (0,0) circle (1mm) (1,1) circle (1mm) (1.5,0) circle (1mm) (1,-1) circle (1mm) (0,-1.5) circle (1mm)  (-1,0) circle (1mm) (-2,0) circle (1mm) (-1.5,1) circle (1mm) (-1.5,-1) circle (1mm);
		\end{tikzpicture}
&\qquad&	\begin{tikzpicture}[style=thick, scale=0.55]
		\draw (0,1.5) -- (0,0) -- (1,1);
		\draw  (1.5,0) -- (0,0);
		\draw  (1,-1) -- (0,0);
		\draw  (0,-1.5) -- (0,0) -- (-2,0) -- (-3,0);				
		\draw  (-2.5,1) -- (-2,0) -- (-2.5,-1);
		\draw[fill=white] (0,1.5) circle (1mm)  (0,0) circle (1mm) (1,1) circle (1mm) (1.5,0) circle (1mm) (1,-1) circle (1mm) (0,-1.5) circle (1mm)  (-2,0) circle (1mm) (-1,0) circle (1mm) (-3,0) circle (1mm) (-2.5,1) circle (1mm) (-2.5,-1) circle (1mm);
		\end{tikzpicture}
&\qquad&	\begin{tikzpicture}[style=thick, scale=0.55]
		\draw (0,1.5) -- (0,0) -- (1,1);
		\draw  (1.5,0) -- (0,0);
		\draw  (1,-1) -- (0,0);
		\draw  (0,-1.5) -- (0,0) -- (-3,0) -- (-4,0);				
		\draw  (-3.5,1) -- (-3,0) -- (-3.5,-1);
		\draw[fill=white] (0,1.5) circle (1mm)  (0,0) circle (1mm) (1,1) circle (1mm) (1.5,0) circle (1mm) (1,-1) circle (1mm) (0,-1.5) circle (1mm)  (-1,0) circle (1mm) (-2,0) circle (1mm) (-3,0) circle (1mm) (-4,0) circle (1mm) (-3.5,1) circle (1mm) (-3.5,-1) circle (1mm);
		\end{tikzpicture}\\
$S^2_{3,5}$ && $S^3_{3,5}$ && $S^4_{3,5}$\\
	\end{tabular}
\end{center}
More formally, for $k,m,n\ge1$, we define $S^k_{m,n}$ to be the graph obtained from a path on $k$ vertices by adding $m$ and $n$ leaves to each of the terminal vertices of the path. Note that type 1 graphs from Corollary~\ref{qc=2} 
have $k$ equal to $2$, $3$ or $4$.   

\begin{proposition}\label{qSkmn}
For $k\ge2$ and $m,n\ge1$, we have $$q(S^k_{m,n})=k+2.$$
\end{proposition}
\begin{proof}
The diameter of $S^k_{m,n}$ is attained by a unique shortest path consisting of $k+1$ edges, hence $q(S^k_{m,n})\geq k+2$ by~\eqref{eq:Fallat}. The maximum multiplicity $M$ of this tree is equal to its path cover number~$\mu$, by~\cite{johnson1999maximum}. If $k>2$, then $P=m+n-1$, hence the minimum rank is $k+m+n-P=k+1$, so $q(S^k_{m,n})\leq k+2$ and we have equality.

The case $k=2$ requires a different approach, because $S^2_{m,n}$ has minimum rank $4$ for $m,n\ge2$, so the previous argument only yields an upper bound of $5$ for $q(S^2_{m,n})$. To see that in fact $q(S^2_{m,n})=4$, consider the matrix 
\[
  A=\begin{pmatrix}
    1 & \sqrt{2} & 0 & 0 \\
    \sqrt{2} & 0 & 1 & 0 \\
    0 & 1 & 0 & \sqrt{2} \\
    0 & 0 & \sqrt{2} & -1 
  \end{pmatrix}.
\]
We have $A\in S(P_4)$ and $\sigma(A)=\{\pm1,\pm\sqrt5\}$, so the diagonal entries of~$A$ corresponding to the leaves of $P_4$, namely $1$ and $-1$, are in $\sigma(A)$. Duplicating these leaves $m-1$ and $n-1$ times, respectively, Lemma~\ref{lem:dup} yields a matrix in $S(S^2_{m,n})$ with the same set of distinct eigenvalues as $A$, hence $q(S^2_{m,n})=4$.
\end{proof}

\begin{lemma}\label{lem:exceptions}
  If $k\ge 1$,  $M\geq m$ and $N\geq n$, then $$q((S^k_{M,N})^c)\leq q((S^k_{m,n})^c).$$
\end{lemma}
\begin{proof}
  Let $v$ be a leaf of $S^k_{m,n}$ contained in an induced
  subgraph $K_{m,1}$. Since $S^{k}_{m+1,n}=\dup(S^k_{m,n},v)$, we have
  \[q((S^k_{m+1,n})^c)=q((\dup(S^k_{m,n},v))^c)\leq q((S^k_{m,n})^c)\]
  by Lemma~\ref{prop:dup}.
  The result follows   by induction and symmetry.
\end{proof}

\begin{proposition}\label{prop:S4}
  We have $q((S^4_{m,n})^c)=2$ for all $m,n\ge1$.
\end{proposition}
\begin{proof}
  By Lemma~\ref{lem:exceptions}, it suffices to establish this for
  $(m,n)=(1,1)$, i.e., to prove that $q(P_6^c)=2$. For this, observe that for
$$A=\npmatrix{
    3 & 2 &  \sqrt{6} \\
    2 & 0 & -2 \sqrt{6} \\
     \sqrt{6} & -2 \sqrt{6} & 1
   }\text{
 and }B=\npmatrix{
     3  & 2 & 0 \\
     2  & 0 & 0 \\
     0 & 0 & -1
  },$$
  the following matrix is orthogonal and lies in $S(P_6^c)$:
  $$X=2^{-5/2}
\npmatrix{
  A & B \\ B & -A
}.
$$
(Alternatively, this follows from~\cite[Corollary~6.9]{2017arXiv170801821B} since $P_6^c=G188$.)
\end{proof}

\begin{proposition}\label{prop:S3}
We have \[q((S_{m,n}^3)^c)=\begin{cases}   
    3,&\text{$\min\{m,n\}=1$,}\\
    2,&m,n\ge2.
  \end{cases}\]
\end{proposition}
\begin{proof}
  We have $q((S_{m,n}^3)^c))\le q((S^3_{1,1})^c)=q(P_5^c)=3$ by Lemma~\ref{lem:exceptions}. 
  
 The graph $(S_{m,1}^3)^c$ consists of a clique on $m+3$ vertices with two edges $\{v,y\},\{w,y\}$ removed, together with an extra vertex~$x$ joined only to $y$ and~$w$. The path $\{x,w\}, \{w,v\}$ is the 
 unique shortest path from $x$ to $v$ in  $(S_{m,1}^3)^c$, hence~\eqref{eq:Fallat} gives $q((S_{m,1}^3)^c)\ge3$.
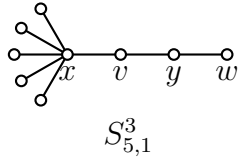
\begin{center}
\begin{tabular}{c}
 \begin{tikzpicture}[style=thick, scale=0.7]
	\draw \foreach \x in {0,1,2} {
				(\x,0) -- (\x + 1,0) };
	\draw \foreach \x in {120,150,...,240} {
				(\x:1) -- (0,0) };
	\draw[fill=white] \foreach \x in {120,150,...,240} {
		                    (\x:1) circle (1mm)
		};		
	\draw[fill=white] \foreach \x in {0,1,2,3} {
		                    (\x,0) circle (1mm)
		};
	\node [below] at (0,0) {$x$};		
	\node [below] at (1,0) {$v$};		
	\node [below] at (2,0) {$y$};		
	\node [below] at (3,0) {$w$};		
 \end{tikzpicture}\\
 $S^3_{5,1}$
 \end{tabular}
 \end{center}

  For $m=n=2$, we can find an orthogonal matrix $X\in S((S_{2,2}^3)^c)$ such as
  \[X=\frac16
   \left(
   \begin{array}{ccccccc}
    -1 & 1 &\alpha_{+} & \alpha_{-} & 2 \sqrt{2} & 3 & 0
      \\
    1 & -1 & \alpha_{-} &\alpha_{+} & -2 \sqrt{2} & 3 &
      0 \\
   \alpha_{+} & \alpha_{-} & -1 & 1 & 2 \sqrt{2} & 0 & 3
      \\
    \alpha_{-} &\alpha_{+} & 1 & -1 & -2 \sqrt{2} & 0 &
      3 \\
    2 \sqrt{2} & -2 \sqrt{2} & 2 \sqrt{2} & -2 \sqrt{2} & -2 & 0 & 0 \\
    3 & 3 & 0 & 0 & 0 & 0 & -3 \sqrt{2} \\
    0 & 0 & 3 & 3 & 0 & -3 \sqrt{2} & 0
   \end{array}
   \right)\] where $\alpha_{\pm}=\tfrac32\sqrt2\pm2$.
 Hence $q((S_{2,2}^3)^c)=2$, and the result follows by Lemma~\ref{lem:exceptions}.
\end{proof}

\begin{proposition}\label{prop:S2}
We have $$q((S_{m,n}^2)^c)=
  \begin{cases}
    4,&m=n=1\\
    3,&\text{$m\ne n$ and $\min\{m,n\}=1$}\\
    2,&m,n\ge2.
  \end{cases}$$
\end{proposition}

\begin{proof}
Note that  $(S_{1,1}^2)^c=P_4$, so $q((S_{1,1}^2)^c)=4$. 

 The graph $(S_{m,1}^2)^c$, $m \geq 2$, is connected and has a leaf, so has a unique shortest path of length 2. Hence $q((S_{m,1}^2)^c)\geq 3$
 by~\eqref{eq:Fallat}. In the case $m=2$, the graph $(S_{2,1}^2)^c$ contains a $4$-cycle, thus  $q((S_{2,1}^2)^c)\leq |(S_{2,1}^2)^c|-2=3$
 by~\eqref{eq:cyclebound}. The result for arbitrary $m\ge2$ follows from Lemma~\ref{lem:exceptions}.

Lemma~\ref{lem:exceptions} allows us to reduce the $m,n \geq 2$ case to
  $(m,n)=(2,2)$, i.e., it is enough to prove that $q(G)=2$ for $G=(S_{2,2}^2)^c$. This
  follows by considering the orthogonal matrix $X\in S(G)$ given by 
    $$X=\frac1{\sqrt3}\npmatrix{
  A & B \\ B & -A
}
, \,A=\npmatrix{
    1&1&1\\1&0&-1\\1&-1&0
} \text{
 and }B=\frac1{\sqrt{2}}\npmatrix{
      0 & 0 & 0\\
      0 & 1 & -1\\
      0 & -1 & 1
 }.$$
Note that $(S_{2,2}^2)^c=G181$, so the same equality is also proven in~\cite[Lemma~6.12]{2017arXiv170801821B}. 
\end{proof}

\subsection{Type 2 exceptional trees}

Type 2 exceptional trees are precisesly the rooted trees~$T$ with depth 2 which are not equal to $K_{1,|T|-1}$. So they consist of a root vertex (which we label as $0$) with $k+l$ neighbours labelled as $1,2,\dots,k,k+1,\dots,k+l$, where $k\geq 1$, $l\geq 0$ and $(k,l)\ne (1,0)$, so that vertices $k+1,\dots,k+l$ are leaves of~$T$, and for $1\leq i\leq k$, vertex $i$ has neighbours consisting of $0$ and $d_i>0$ further leaves of~$T$. Note that $|T|=1+k+l+\sum_{i=1}^k d_i$. Let us denote this tree by $T=W(k,l,\delta)$ where $\delta=(d_1,\dots,d_k)\in \mathbb N^k$. We also write $W(k,l)=W(k,l,(1,\dots,1))$; this tree consists of $k$ copies of $P_3$ and $l$ copies of $P_2$, all joined at a common end vertex. Finally, let $\delta_{\min}$ denote the minimal element in $\delta$. 
 \begin{center}
 \begin{tabular}{ccc}
	\begin{tikzpicture}[style=thick, scale=0.7]
		\draw \foreach \x in {4,4.5,5} {
				(\x,0) -- (4.5,1) };
		\draw \foreach \x in {5.5,6,6.5} {
				(\x,0) -- (6,1) };
		\draw \foreach \x in {4.5,6,7,8,9} {
				(\x,1) -- (6,2) };
		\draw (7,0)--(7,1);
		\draw[fill=white] \foreach \x in {4,4.5,5,5.5,6,6.5,7} {
		                    (\x,0) circle (1mm)
		};	
		\draw[fill=white] \foreach \x in {4.5,6,7,8,9} {
		                    (\x,1) circle (1mm)
		};		
		\draw[fill=white] (6,2) circle (1mm);
		\end{tikzpicture}
		& \qquad \qquad &
	\begin{tikzpicture}[style=thick, scale=0.7]
		\draw \foreach \x in {1,2,3} {
				(\x,0) -- (\x,1) };
		\draw \foreach \x in {1,2,...,5} {
				(\x,1) -- (3,2) };
		\draw[fill=white] \foreach \x in {1,2,...,5} {
		                    (\x,1) circle (1mm)
		};	
		\draw[fill=white] \foreach \x in {1,2,3} {
		                    (\x,0) circle (1mm)
		};		
		\draw[fill=white] (3,2) circle (1mm);
		\end{tikzpicture}
\\
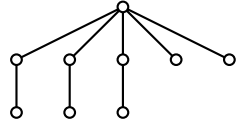
$W(3,2,(3,3,1))$ && $W(3,2)=W(3,2,(1,1,1))$  \\
	\end{tabular}
\end{center}
Observe that $W(1,0,(d_1))=K_{1,d_1+1}$ and $W(1,l,(d_1))=S^2_{l,d_1}$ for $l,d_1\ge1$, so we have established the values of $q$ for these graphs and their complements above in Proposition~\ref{qSkmn} and Proposition~\ref{prop:S2}, and we do not need to consider these graphs further.

\begin{proposition}\label{prop:qW}
  We have $q(W(k,l,\delta))=5$ for all $k\ge2$, all $l\ge0$
  and all $\delta\in \mathbb{N}^k$.
\end{proposition}
\begin{proof}
  Note that for $k\ge2$, the graph $G=W(k,l,\delta)$ contains a unique shortest path of four edges, so
  $q(G)\ge5$ by~\eqref{eq:Fallat}. Moreover, $G$ may be obtained from either $W(k,0)$ or $W(k,1)$ (depending on whether or not $l=0$) by repeated duplication of leaves. By Lemma~\ref{lem:dup}, it suffices to find $X\in S(W(k,0))$ and
  $Y\in S(W(k,1))$ so that the diagonal entries of $X$ and of~$Y$ are contained in $\sigma(X)$ and
  $\sigma(Y)$, respectively, and $q(X)=q(Y)=5$. For this, take $c=(1,1,\dots,1)^T\in \mathbb{R}^k$, define $v=(\tfrac3k)^{1/2}c$, $w=(\tfrac4k)^{1/2}c$,
  \[ X=
    \begin{pmatrix}
      0&I_k&0\\
      I_k&0&v\\
      0&v^T&0
    \end{pmatrix}\quad\text{and}\quad
     Y=
    \begin{pmatrix}
      0&I_k&0&0\\
      I_k&0&w&0\\
      0&w^T&1&1\\
      0&0&1&1
    \end{pmatrix}\]
  and note that the sets of distinct eigenvalues of~$X$ and of~$Y$ are $\{0,\pm1,\pm2\}$ and $\{0,\pm1,-2,3\}$, respectively.
\end{proof}

We now turn to the calculation of $q(W(k,l,\delta)^c)$. Lemma~\ref{prop:dup} immediately yields:
\begin{lemma}\label{lem:W}
  For any $k\geq 1$, $l\geq 0$ and $\delta \in \mathbb N^k$, we have \[q(W(k,l,\delta)^c)\leq q(W(k,\min\{1,l\})^c).\]
\end{lemma}

We will determine $q(W(k,l,\delta)^c)$ by considering the cases $k=2$ and $k\ge3$ in turn.

\begin{proposition} \label{prop:W2c} 
For $l\ge0$ and  $\delta \in \mathbb N^2$, we have
\[ q(W(2,l,\delta)^c)=
  \begin{cases}
    3,& \delta_{min}=1,\\
    2,&\text{otherwise.}
  \end{cases}\]
\end{proposition}

\begin{proof}
We have $W(2,0,(d_1,d_2))=S^3_{d_1,d_2}$, so the statement for $l=0$ was established in Proposition~\ref{prop:S3}.
Moreover, for any $l\ge0$ and $d_2\ge1$, the graph $W(2,l,(1,d_2))^c$ contains a unique shortest path on $2$ edges, from vertex~$0$ to vertex~$2$ (via vertex $3+l$), and $W(2,1)^c=G184$ has $q=3$ by~\cite[Corollary 6.2]{2017arXiv170801821B}.
Hence \[3\leq q(W(2,l,(1,d_2))^c)\leq q(W(2,\min\{1,l\})^c)=3\]
by   \eqref{eq:Fallat}  and Lemma~\ref{lem:W}, so $q(W(2,l,(1,d_2))^c)=3$ for $d_2\geq 1$ and $l\ge0$. On the other hand, the following orthogonal matrix shows that $q(W(2,1,(2,2))^c)=2$, where $\alpha_{\pm}=\tfrac12(3\pm\sqrt6)$:
\[
\frac16
\left(
\begin{array}{cccccccc}
 0  & -3        & 3         & -3        & 3         & 0                        & 0                        & 0          \\
 -3 & \alpha_+  & -\alpha_- & -\alpha_+ & \alpha_-  & \sqrt{6}                 & 0                        & \sqrt{6}   \\
 3  & -\alpha_- & \alpha_+  & \alpha_-  & -\alpha_+ & \sqrt{6}                 & 0                        & \sqrt{6}   \\
 -3 & -\alpha_+ & \alpha_-  & \alpha_+  & -\alpha_- & 0                        & \sqrt{6}                 & \sqrt{6}   \\
 3  & \alpha_-  & -\alpha_+ & -\alpha_- & \alpha_+  & 0                        & \sqrt{6}                 & \sqrt{6}   \\
 0  & \sqrt{6}  & \sqrt{6}  & 0         & 0         & -2-\sqrt{6}              & -2+\sqrt{6}              & 2          \\
 0  & 0         & 0         & \sqrt{6}  & \sqrt{6}  & -2+\sqrt{6}              & -2-\sqrt{6}              & 2          \\
 0  & \sqrt{6}  & \sqrt{6}  & \sqrt{6}  & \sqrt{6}  & 2                        & 2                        & -2
\end{array}
\right).\]
By Lemma~\ref{lem:W}, $q(W(2,l,(d_1,d_2))^c)=2$ for $d_1,d_2\geq 2$.
\end{proof}

\begin{proposition}\label{prop:W3c}
  $q(W(k,l,\delta)^c)=2$ for any $k \geq 3$, $l\ge0$, and $\delta \in \mathbb N^k$.
\end{proposition}
\begin{proof}
 By Lemma~\ref{lem:W}, it suffices to establish this for $\delta=(1,1,\ldots, 1)$ and $l\in \{0,1\}$, i.e., for the graph $G=W(k,l)$ where $l\in \{0,1\}$.
 Let us write $n=|G|=2k+l+1$ for the number of vertices in~$G$.

 Let $k+1$ and $n-k-1$ be the
  numbers of black and white vertices, respectively, in the
  $2$-colouring of the graph $G$ for which the root vertex is
  coloured black. Relabel the vertices of~$G$ as $0,\dots,n$, with
  root vertex~$0$ and black vertices $0,\dots,k$. 
  
We are looking for a matrix~$B$ and a vector $v$ satisfying the hypotheses of Corollary~\ref{cor:sqrt}, so that the matrix $\Mhat{\alpha}{v}$ defined in~\eqref{eq:N} lies in $S(W(k,l)^c)$ for some $\alpha\in[-1,1]$. We need $B$ to be a $k\times (k+l)$ matrix with zeros precisely on the diagonal and for which $BB^T$ and $B^TB$ are both generalised irreducible; furthermore we require $v\in \R^k$ with no zero entries to satisfy $B^Tv=0$.

Such matrices are easy to find: consider, for example, the $k\times (k+l)$ matrix~$B=[b_{i,j}]$ with zeros on the diagonal and $1$s in every off-diagonal entry except $b_{k,1}=2-k$, and, in the case $l=1$, with
$b_{k-1,k+1}=\tfrac23$ and $b_{k,k+1}=\tfrac13$. Examples for $k=4$ and $l=0$, and $k=4$ and $l=1$ are given below: 
$$\left(
\begin{array}{cccc}
 0 & 1 & 1 & 1 \\
 1 & 0 & 1 & 1 \\
 1 & 1 & 0 & 1 \\
 -2 & 1 & 1 & 0 \\
\end{array}
\right), \quad\left(
\begin{array}{ccccc}
 0 & 1 & 1 & 1 & 1 \\
 1 & 0 & 1 & 1 & 1 \\
 1 & 1 & 0 & 1 & \frac{2}{3} \\
 -2 & 1 & 1 & 0 & \frac{1}{3} \\
\end{array}
\right).$$
Then $B^Tv=0$ for $v=(2-k, 1, \ldots , 1)^T \in \R^k$, and the matrices $BB^T$ and $B^TB$ both have no zero entries and so, trivially, are generalised irreducible.  After rescaling to make the spectral norm of $B$ less than one, the matrices satisfy the conditions of Corollary~\ref{cor:sqrt}, so $q(G^c)=2$.
\end{proof}

\section{Complements of graphs with \texorpdfstring{$q(G)\ge|G|-1$}{q(G)≥|G|-1}}\label{q=n-1}

Since Conjecture~\ref{conjecture} posits an upper bound on $q(G)+q(G^c)$, it is natural to consider the case when $q(G)$ (or $q(G^c)$) alone is large as one might hope to find a counterexample with this property. In this section, we consider graphs with $q(G)\geq |G|-1$. Such graphs have been characterised in~\cite{MR3665573}, and here we compute $q$ for their complements. This will allow us to show in the following section that they do in fact satisfy the conjecture.

\label{sect:n-1}
\begin{theorem}[{\cite[{Theorem 4.14}]{MR3665573}}]\label{thm:fallat}
A graph $G$ has $q(G)\geq |G| -1$ if and only if $G$ is one of the following: 
\begin{enumerate}
\item a path,
\item the disjoint union of a path and an isolated vertex,
\item a path with one leaf attached to an interior vertex,
\item a path with an extra edge joining two vertices at distance 2.
\end{enumerate}
\end{theorem}

Since the graphs in items 1 and 3 are trees, the values of $q$ for their complements are covered in Section~\ref{sec:trees}.

Let $R_{n}$ denote the complement of $P_{n-1}\cup \{n\}$, a typical graph listed in item 2.
\begin{proposition}\label{prop:Rn}
  We have
$$q(R_n)= \begin{cases}
    2,& n=2  \text{ or }  n\geq 6,\\
    3,& 3 \leq n \leq 5.
  \end{cases}
  $$
  In particular, the graphs $R_n$ and $R_n^c$ satisfy
  inequality~\eqref{eq:conjecture} for all $n\in\mathbb N$.
\end{proposition}

\begin{proof}
  For $n=2,3$ we have $R_n=P_{n}$, so $q(R_2)=2$ and $q(R_3)=3$.

  Since $R_4=\jdup(P_3,v)$ where $v$ is one of the terminal vertices of $P_3$, we have $q(R_4)\le q(P_3)=3$ by 
  Lemma~\ref{prop:dup}. Moreover, $R_4$ contains a unique shortest path %
  of length $2$, %
  so $q(R_4)=3$ by \eqref{eq:Fallat}.

  We have $R_5=P_4\vee \{5\}$, and it is easy to see that the terminal vertices of the copy of $P_4$ are degree $2$ in $R_5$ and have one neighbour in common, so $q(R_5)>2$ by~\cite[Corollary~4.5]{MR3118943}. By Theorem~\ref{thm:fallat}, $q(R_5)<4$, so $q(R_5)=3$. 
  For $n\ge6$, it is easy to check that
  $R_n^c$ satisfies the hypotheses of Theorem~\ref{thm:bipartite}, so $q(R_n)=2$. Since $q(R_n^c)=n-1$ by Theorem~\ref{thm:fallat}, inequality~\eqref{eq:conjecture} follows.
\end{proof}

Now consider a typical graph under item~4, namely the graph
$T_{m,n}$ on $m+n+1$ vertices consisting of a $K_3$ joined to a $P_m$ and a $P_n$ at two different vertices (terminal vertices of the paths), where $m\geq n\ge1$. Equivalently, $T_{m,n}$ is a $P_{m+n+1}$ with the extra edge $\{m,m+2\}$ added.

 \begin{center}
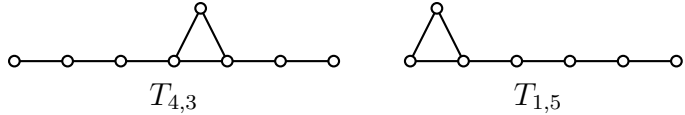
 \begin{tabular}{ccc}
	\begin{tikzpicture}[style=thick, scale=0.7]
		\draw \foreach \x in {0,1,...,5} {
				(\x,0) -- (\x + 1,0) };
		\draw (3,0)--(3.5,1)--(4,0);
		\draw[fill=white] \foreach \x in {0,1,...,6} {
		                    (\x,0) circle (1mm)
		};		
		\draw[fill=white] (3.5,1) circle (1mm);
		\end{tikzpicture}
&\qquad&	\	\begin{tikzpicture}[style=thick, scale=0.7]
		\draw \foreach \x in {0,1,...,4} {
				(\x,0) -- (\x + 1,0) };
		\draw (0,0)--(0.5,1)--(1,0);
		\draw[fill=white] \foreach \x in {0,1,...,5} {
		                    (\x,0) circle (1mm)
		};		
		\draw[fill=white] (0.5,1) circle (1mm);
		\end{tikzpicture}
\\
$T_{4,3}$ && $T_{1,5}$\\
	\end{tabular}
\end{center}
Note that $T_{2,2}$ is one of the four graphs singled out in Figure~\ref{fig:exceptions}.

\begin{lemma}\label{lem:edgedeletion}
  If $e$ is an edge in a graph~$H$ and $G=H\setminus\{e\}$, then
  $q_M(G^c)\leq q_M(H^c)$ and $q_S(G^c)\leq q_S(H^c)$.
\end{lemma}
\begin{proof}
  Since $H^c=G^c\setminus\{e\}$ is a spanning subgraph of $G^c$, this is immediate by
  Theorem~\ref{thm:smp}. 
\end{proof}

\begin{proposition}\label{lem:m+n=6}
  If $m,n\ge1$ with $m+n=6$, then $q(T_{m,n}^c)=q_S(T_{m,n}^c)=2$.
\end{proposition}
\begin{proof}
Let $H$ be the graph obtained by adding an edge to $C_7$ connecting two vertices at distance~$2$. Under the assumption $m+n=6$, the graph $T_{m,n}$ may be obtained from $H$ by deleting a single edge, so $q_M(T_{m,n}^c)\leq q_M(H^c)$ by Lemma~\ref{lem:edgedeletion}. It therefore suffices to exhibit an orthogonal matrix in $S(H^c)$ with the SMP (equivalently, with the SSP), such as:
\[ A= \frac16\left(
\begin{array}{ccccccc}
 0 & 0 & 4 & 0 & 2 \sqrt{2} & 2 \sqrt{3} & 0 \\
 0 & -3 & -\sqrt{6} & 3 & 2 \sqrt{3} & 0 & 0 \\
 4 & -\sqrt{6} & 0 & -\sqrt{6} & 0 & 0 & 2 \sqrt{2} \\
 0 & 3 & -\sqrt{6} & 1 & 0 & 2 \sqrt{2} & 2 \sqrt{3} \\
 2 \sqrt{2} & 2 \sqrt{3} & 0 & 0 & 3 & -\sqrt{6} & -1 \\
 2 \sqrt{3} & 0 & 0 & 2 \sqrt{2} & -\sqrt{6} & 2 & -\sqrt{6} \\
 0 & 0 & 2 \sqrt{2} & 2 \sqrt{3} & -1 & -\sqrt{6} & 3 \\
\end{array}
\right).\qedhere\]
\end{proof}

\begin{theorem}\label{thm:Tmnc}
  We have $q_{S}(T_{m,n}^c)=2$ for any $m, n\ge1$ with $m+n\ge6$.
\end{theorem}
\begin{proof}
We proceed by induction on $N=m+n+1$. The base case, $N=7$, is covered in Proposition~\ref{lem:m+n=6}.  

Suppose $M \in S(T_{m,n}^c)$ has the SSP, with $q(M)=2$ where $m\ge3$. Let $H_{m,n}$ be the graph on $N+1$ vertices obtained from $T_{m,n}$ by duplicating the second vertex in the path $P_m$. 
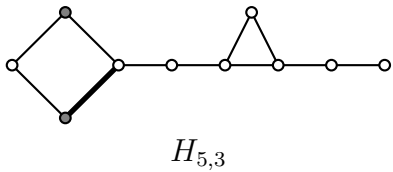
 \begin{center}
 \begin{tabular}{c}
	\begin{tikzpicture}[style=thick, scale=0.7]
		\draw \foreach \x in {1,...,5} {
				(\x,0) -- (\x + 1,0) };
		\draw (3,0)--(3.5,1)--(4,0);
		\draw (1,0)--(0,1)--(-1,0)--(0,-1);
                \draw [line width=2pt](0,-1)--(1,0);
		\draw[fill=white] \foreach \x in {-1,1,2,...,6} {
		                    (\x,0) circle (1mm)
		};		
		\draw[fill=white] (3.5,1) circle (1mm);
		\draw[fill=gray] (0,1) circle (1mm) (0,-1) circle (1mm);
		\end{tikzpicture}
\\
$H_{5,3}$ \\
	\end{tabular}
\end{center}
Since $T_{m+1,n}$ may be obtained from $H_{m,n}$ by deleting one edge (corresponding to the edge marked in bold for $H_{5,3}$ in the diagram above), to complete the induction step it suffices by Lemma~\ref{lem:edgedeletion} to construct a matrix $\hat M \in S(H_{m,n}^c)$ with $q(\hat M)=2$ which has the SSP (or equivalently, the SMP).

We will use Lemma~\ref{HS04} to construct~$\hat M$, showing in addition that~$\hat M$ has the SSP. Let us order the vertices in $T_{m,n}$ so that vertex $N-1$ is the leaf of $T_{m,n}$ in $P_m$ at distance $m-1$ from a degree~$3$ vertex, and the vertices with distance $1$, $2$ and $3$ from this leaf are $N$, $N-2$ and $N-3$, respectively. We adopt the same numbering for the corresponding vertices in $T_{m,n}^c$. By our inductive hypothesis, there is a matrix $M$ with the SSP and $q(M)=2$ of the form
 $$M=\npmatrix{A & b \\ b^T & \alpha}\in S(T_{m,n}^c),$$
 where $M$ matches the ordering of vertices in $T_{m,n}^c$ established above, hence  $b=\npmatrix{b_1 & b_2 & \ldots & b_{N-3} & 0 & 0}^T $.  By~\cite[Lemma~2.3]{MR3118943},  we may assume without loss of generality that $-1$ and~$1$ are the two distinct eigenvalues of $M$. The graph $T_{m,n}^c$ is connected, so $b\ne 0$. Since $M$ is orthogonal, we have $\alpha\not\in \{-1,1\}$.
Let 
$$B=\frac{1}{2}\left(
\begin{array}{cc}
 \alpha +1& \alpha -1 \\
 \alpha -1 & \alpha +1 \\
\end{array}
\right).$$
Note that $B$ has no zero entries, and has eigenvalues $\alpha$ and $1$ with the corresponding eigenvectors 
$$u=\frac{1}{\sqrt{2}}\npmatrix{1 \\ 1}\qand v=\frac{1}{\sqrt{2}}\npmatrix{1 \\ -1}.$$
By Lemma~\ref{HS04}, the matrix 
 $$\hat M=\npmatrix{A & bu^T \\ ub^T & B}  \in S(H_{m,n}^c)$$
 has \[\sigma(\hat M)=\sigma(M)\cup(\sigma(B)\setminus\{\alpha\})=\sigma(M)\cup \{1\},\] so 
 $q(\hat M)=2$. It only remains to check that $\hat M$ has the SSP. 
 
 Since $M$ has the SSP we know that for an $(N-1)\times (N-1)$ real
 matrix $X$ and $y \in \R^{N-1}$, the following conditions imply that
 $X=0$ and $y=0$:
 \begin{gather*}
 X \circ A=0,\quad X \circ I_{n-1}=0,\quad y \circ b=0, \\
 [A,X]+by^T-yb^T=0,\quad
 Ay-Xb-\alpha y=0.
 \end{gather*}
 Suppose $$Z=\npmatrix{X & Y \\ Y^T & 0},$$ with 
\[Z \circ \hat M=0,\quad Z \circ I_{n+1}=0\qand [\hat M,Z]=0. \] The condition $[\hat M,Z]=0$ gives us the following equalities: 
\begin{align}
[A,X]+bu^TY-Yub^T&=0 \label{MSSP1} \\
AY-Xbu^T-YB&=0. \label{MSSP2}
\end{align}
Now let us consider $\hat y=Yu.$ Note that $X\circ A=0,$ $X \circ I_{N-1}=0$  and $\hat y \circ b=0$. Furthermore, (\ref{MSSP1}) can be written as
$[A,X]+b\hat y^T-\hat y b^T=0,$ and (\ref{MSSP2}) multiplied by $u$ gives us 
$A \hat y-Xb-\alpha \hat y=0$. Since $M$ has the SSP, we conclude that $X=0$ and $\hat y=0.$ At this point we are left with the following conditions:
$$AY-YB=0, \, Y \circ bu^T=0 \qand Yu=0.$$ 
The conditions $Y \circ bu^T=0 \text{ and }Yu=0$ imply that $Y$ is of the form
 $$Y=\npmatrix{0_{N-3,1} \\ y_1 \\ y_2}\npmatrix{1 & -1}.$$
Now we have $YB=Y,$ and 
 $AY=(y_1 c_{N-2}(A)+y_2c_{N-1}(A))\npmatrix{1 & -1},$ where 
  $c_i(A)$ denotes the $i$-th column of $A$. The condition $AY-YB=0$ reduces to the linear system
   $$y_1 c_{N-2}(A)+y_2c_{N-1}(A)=\npmatrix{0_{N-3,1} \\ y_1 \\ y_2}.$$
   Because $M\in S(T_{m,n}^c)$, we have $A\in S(T_{m,n}^c\setminus\{N\})$, from which it follows that the $(N-3)$-rd coordinate of $c_{N-2}(A)$ is $0$ and the $(N-3)$-rd coordinate of $c_{N-1}(A)$ is not zero. Hence, the only solution to this linear system is $y_1=y_2=0$. Therefore $Y=0$  and $\hat{M}$ has the SSP.
\end{proof}

\begin{corollary}\label{cor:Tc} For $m\ge n\ge1$, we have
  \[q(T_{m,n}^c)=
    \begin{cases}
      1,&(m,n)=(1,1),\\
      4,&(m,n)=(2,2),\\
      3,&(m,n)\in \{(2,1),(3,1),(4,1),(3,2)\},\\
      2,&\text{otherwise.}
    \end{cases}\] In particular, for $(m,n)\ne (2,2)$, the graphs
  $T_{m,n}$ and $T_{m,n}^c$ satisfy inequality~\eqref{eq:conjecture}.
\end{corollary}
\begin{proof}
We have
\begin{itemize}
\item  $T_{1,1}^c=K_3^c=3 K_1$, so $q(T_{1,1}^c)=1$;
\item  $T_{2,2}^c=T_{2,2}$, so $q(T_{2,2}^c)=4$ by Theorem~\ref{thm:fallat};
\item  $T_{2,1}^c=P_3\cup \{v\}$, so $q(T_{2,1}^c)=3$ by Theorem~\ref{thm:fallat};
\item  $T_{3,1}^c$ is the Banner graph Bnr of~\cite{2017arXiv170800064B}, which is shown there to have $q(T_{3,1}^c)=3$; and
\item $T_{3,2}^c=G166$ and $T_{4,1}^c=G173$, which both have $q=q_S=2$ by~\cite[Corollary~6.2]{2017arXiv170801821B}.
\end{itemize}
The remaining cases %
  have $m+n \geq 6$ and therefore $q(T_{m,n}^c)=2$ by Proposition~\ref{thm:Tmnc}. Since $q(T_{m,n})=|T_{m,n}|-1$ by Theorem~\ref{thm:fallat}, inequality~\eqref{eq:conjecture} follows for $(m,n)\ne(2,2)$.
\end{proof}

\section{On the Conjecture}\label{sec:conjecture}

In this section we prove that Conjecture~\ref{conjecture} holds for several families of graphs.
First, note that results from Section~\ref{sec:trees} on type 1 and type 2 exceptional trees imply the following. 

\begin{proposition}\label{prop:qqcW}
  \begin{enumerate}
  \item For $k\in \{2,3,4\}$ and $m,n\in \mathbb N$, we have
\[q(S^k_{m,n})+q((S^k_{m,n})^c)=
  \begin{cases}
   6, & k=2,\; m,n \geq 2,\\[\smallskipamount]
   7, & \parbox[t]{.8\textwidth}{$k=2,\; \min\{m,n\}=1,\;m\ne n$, or \\  $k=3,\; m,n \geq 2$,\smallskip}\\
   8, & \parbox[t]{.8\textwidth}{$k=2,\; m=n=1$, or\\$k=3,\;  \min\{m,n\}=1$, or $k=4$,}
  \end{cases}\]
so with the sole exceptions of $(k,m,n)\in \{(2,1,1),(3,1,1)\}$, corresponding to $S^k_{m,n}\in \{P_4,P_5\}$, we have $q(S^k_{m,n})+q((S^k_{m,n})^c)\le |S^k_{m,n}|+2$.
\item  For $k\ge2$ and $l\ge0$ and $\delta\in \mathbb N^k$, we have
\[q(W(k,l,\delta))+q(W(k,l,\delta)^c)=
  \begin{cases}
   7, & k=2,\; \delta_{\min}\geq 2  \text{ or }  k\geq 3,\\
   8, & k=2,\; \delta_{\min}=1,
 \end{cases}\] so with the sole exception of $(k,l,\delta)=(2,0,(1,1))$, corresponding to $W(k,l,\delta)=P_5$, we have $q(W(k,l,\delta))+q(W(k,l,\delta)^c)\le |W(k,l,\delta)|+2$.
\end{enumerate}
\end{proposition}

We now deduce that the conjecture holds for any tree.

\begin{theorem}\label{cor:tree}
  If $T$ is any tree except $P_4$, $P_5$, then  $q(T)+q(T^c)\leq |T|+2$.
 \end{theorem} 
 \begin{proof}
   Since $q(T)\le |T|$, the inequality holds if
   $q(T^c)=2$. 
   In all other cases, the inequality follows from
   Theorem~\ref{thm:treeC} and Proposition~\ref{prop:qqcW}.
 \end{proof}

\begin{theorem}\label{thm:n-1}
   If $G$ is any graph %
  with $q(G)\geq|G|-1$ except for $P_4$, $P_5$ and $T_{2,2}$, then $q(G)+q(G^c)\leq |G|+2$.
\end{theorem}
\begin{proof}
  This follows by combining Theorem~\ref{thm:fallat} with Theorem~\ref{cor:tree}, Proposition~\ref{prop:Rn} and Corollary~\ref{cor:Tc}.
\end{proof}

\begin{corollary} \label{cor:q leq 4}
If $G$ is any graph except $P_4$, $P_5$, $P_5^c$ or $T_{2,2}$ %
and $q(G) \leq 4$ or $q(G^c) \leq 4$, then 
$ q(G)+q(G^c)\leq |G|+2$. 
\end{corollary}
  \begin{proof}
    We may assume by
    symmetry that $q(G^c)\leq 4$. If $q(G)\leq |G|-2$, then the
    inequality clearly holds. On the other hand, if $q(G)\geq |G|-1$,
    then the inequality holds by Theorem~\ref{thm:n-1}.
  \end{proof}
  
We recall that in~\cite[Theorem~35]{MR3665573}, it was shown that \begin{equation}q(G^c)\le 2\chi(G)\label{eq:chibound}
\end{equation}
for any graph~$G$.
\begin{corollary}\label{cor:bipartite}
Conjecture~\ref{conjecture} holds for the following families of graphs:
\begin{enumerate}
\item bipartite graphs and their complements, including trees;
\item graphs with minimum rank at most three and their complements;
  and
\item joins $G \vee H$ and disjoint unions $G\cup H$ of two graphs,
  such that the number of vertices of $G$ and $H$ differ by at most 2.
\end{enumerate}
\end{corollary}
\begin{proof} The first two assertions follow immediately from
  Corollary~\ref{cor:q leq 4} combined with~\eqref{eq:chibound}, and the
  bound~\eqref{eq:mrbound}, respectively. The third assertion for
  joins follows in a similar way using
  \cite[Theorem~5.2]{monfared2016nowhere}
  and~\cite[Theorem~4.6]{2017arXiv170801821B}, and taking complements
  yields the assertion about disjoint unions.
\end{proof}

Note that the chromatic number $\chi(G)$ of a graph $G$ can be expressed in the terms of
its complement as
$$\chi(G)=\min\{k\colon G^c \supseteq K_{n_1} \cup K_{n_2}\cup \ldots \cup K_{n_k},\; n_1+\ldots+n_k=|G|\},$$
so $\chi(G)=\theta(G^c)$ where $\theta(H)$ is the clique cover number of a graph~$H$.
Here, we generalise this notion to 
$$\theta^*(G, \mathcal{S})=\min\{k\colon G \supseteq 
 H_{1} \cup H_{2}\cup \ldots \cup H_{k},\; \sum_{j=1}^k |H_j|=|G|, \; H_j \in \mathcal{S} \},$$
where $\mathcal S$ is a given family of graphs. As a special case that is relevant to $q(G)$, we define 
$$\theta^*(G, \ell)=\theta^*(G,\mathcal S_{\ell}), \text{ where } \mathcal S_{\ell}=\{G\colon q_M(G) \leq \ell\}.$$
The next lemma is a generalisation of the bound~\eqref{eq:chibound} and its proof from~\cite[Theorem~35]{MR3665573}.

\begin{lemma}\label{chi*}
  For any graph $G$, we have $q(G) \leq \ell \cdot \theta^*(G, \ell)$.
\end{lemma}
\begin{proof}
  Choose graphs $H_1, \ldots, H_{t}\in \mathcal S_\ell$ so that
  $H_{1} \cup H_{2}\cup \ldots \cup H_{t}$ is a spanning subgraph of
  $G$, where $t=\theta^*(G,\ell)$. For each $j=1,2,\ldots,t$, we may
  choose $A_j \in S(H_j)$ with the SMP so that $q(A_j)=q_M(H_j)\le\ell$, and
  so that matrices $A_1,\ldots,A_{t}$ have pairwise disjoint
  spectra. By Theorem~\ref{thm:oplus}, the matrix
  $A=A_1 \oplus A_2 \oplus \ldots \oplus A_{t}$ has the SMP and thus
  by Theorem~\ref{thm:smp} there exists a matrix $B \in S(G)$ with 
  $\sigma(B)=\sigma(A)$. Therefore
  $q(G) \leq q(B)=q(A)\leq \ell \cdot t$.
\end{proof}

Note that Corollary~\ref{cor:q leq 4} and Lemma~\ref{chi*} imply that
Conjecture~\ref{conjecture} holds for graphs~$G$ with
$\theta^*(G, 2)=2$ or $\theta^*(G^c, 2)=2$. While we can envision
applications of Lemma~\ref{chi*} more broadly, we will exploit it here
to prove the conjecture for unicyclic graphs using only the fact that
the graphs $K_n$ and also
$K_n\setminus \{e\}$ for $n\ge4$ belong to $\mc S_2$.

  \begin{theorem}\label{unicyclic}
    If $G$ is any unicyclic graph except $T_{2,2}$, then $q(G^c) \leq 4$ and hence $G$ satisfies Conjecture~\ref{conjecture}.
  \end{theorem}
  
  \begin{proof}
 Let $G\ne T_{2,2}$ be a unicyclic graph, so that $G$ contains a cycle $C_t$ for some unique
 $t\ge3$. If $t$ is even, then $G$ is bipartite and  the statement follows by Corollary~\ref{cor:bipartite}.
  
 If $|G|\le 5$ and $q(G^c)\ge5$, then $q(G^c)=5$ so $G^c=P_5$ and $G=P_5^c$ which is not unicyclic, a contradiction. Now let $G$ be a unicyclic graph with $|G|=6$ and $q(G^c)\geq 5$. Then $G^c$ is one of the graphs listed in Theorem~\ref{thm:fallat}, so it is either a forest or is unicyclic, hence $|E(G^c)|\leq 6$. 
As $G$ is unicyclic, we have $|E(G)|\leq 6$, so
\[
  15 = |E(K_6)| = |E(G)|+|E(G^c)| \le 12,
\] a contradiction.

In the remaining cases, $t$ is odd and $|G| \geq 7$. Then there exists
a 3-colouring of $G$ where $c_1$ vertices are coloured black, $c_2$
vertices are coloured white and one vertex is coloured red; since
$c_1+c_2+1\geq 7$, we may assume that $c_1 \geq 3$. The black vertices
together with the red vertex form $K_{c_1+1}\setminus\{e\}$ in $G^c$,
and the white vertices correspond to $K_{c_2}$ in $G^c$. Hence,
$K_{c_2}\cup (K_{c_1+1}\setminus\{e\})$ is a spanning subgraph in
$G^c$. Since $c_1+1 \ge 4$, we have
$q_M(K_{c_2})\leq q_M(K_{c_1+1}\setminus\{e\})=2$
by~\cite[Corollary~2.8]{2017arXiv170801821B}, and so
$\theta^*(G^c, 2)\leq 2$.  By Lemma~\ref{chi*}, we have $q(G^c)\leq 4$
as desired, and Corollary~\ref{cor:q leq 4} completes the proof.
  \end{proof}

The bound~\eqref{eq:cyclebound} implies that Conjecture~\ref{conjecture} holds if~$G$ is a graph for which both $c(G)$ and $c(G^c)$ are close to $n$. The following conditions are precisely those in which the conjecture is resolved in this way. (We omit the routine proof.)

\begin{proposition}\label{prop:56}
Conjecture~\ref{conjecture} holds for graphs $G$ on $n$ vertices that satisfy one of the following inequalities:
\begin{enumerate}
\item $n$ is odd, $c(G) \geq n-3$ and $c(G^c) \geq n-1$, 
\item $n$ is even and $\min\{c(G),c(G^c)\} \geq n-2$,
\item $n$ is even,  $c(G) \geq n-3$, and $c(G^c)=n$. 
\end{enumerate}
\end{proposition}

Finally, we prove the conjecture for all graphs with at most $7$ vertices. 

\begin{theorem}\label{lemma:67}
  If $G$ is any graph except $P_4$, $P_5$, $P_5^c$ or $T_{2,2}$  with at most 7 vertices, then $q(G)+q(G^c)\leq |G|+2$.
\end{theorem}
  \begin{proof}
    If $|G|\le 6$, then either $q(G)\leq 4$ or $q(G)\geq 5\geq |G|-1$, so
    these cases are covered by Corollary~\ref{cor:q leq 4} and
    Theorem~\ref{thm:n-1}, respectively.
    
     Suppose now that $G$ is a counterexample to the conjecture with
    $7$ vertices. Then $q(G)>4$ and $q(G^c)>4$ by Corollary~\ref{cor:q leq 4}, and by~\eqref{eq:cyclebound} we have $c(G)\le 5$ and $c(G^c)\le 5$, i.e., neither $G$ nor $G^c$ can contain a $6$-cycle or a $7$-cycle as a subgraph.  
    
    A  computer search shows that there are $24$ pairs %
        $\{G,G^c\}$ of
    graphs of order~$7$ for which neither contains a $6$-cycle or a
    $7$-cycle, and that in every case, one graph in the pair, say~$G$,
    contains a pair of joined duplicate vertices $v,w$. Then $G=\jdup(H,v)$ for
    $H=G\setminus\{w\}$ with $|H|=6$, so
    \[5\leq q(G)=q(\jdup(H,v))\leq q(H)\] by
    Lemma~\ref{prop:dup}. It follows easily from Theorem~\ref{thm:fallat} that if $H$ is a graph with $|H|=6$ and $q(H)\ge 5$,
    then $H^c$ must
    contain a $6$-cycle as a subgraph, so $G^c=\dup(H^c,v)$ also contains a
    $6$-cycle; a contradiction.
  \end{proof}

  \begin{remark}
    We have been unable to adapt this proof to larger graphs. If $G$
    were a counterexample with $|G|=8$, then Corollary~\ref{cor:q leq 4} and
    the cycle bound (\ref{eq:cyclebound}) give $c(G)<8$, i.e., $G$
    contains no Hamiltonian cycle. The same holds for $G^c$, and
    Proposition~\ref{prop:56} may then be used to reduce the number of
    cases under consideration to $323$ possible pairs
    $\{G,G^c\}$. However, here the argument appears to break down
    because the gap between $|G|=8$ and our constraint $q(G)\ge 5$ is
    too large for us to apply Theorem~\ref{thm:fallat} in the same
    way as above.
  \end{remark}

\begin{remark}
 If a graph $G$ with at least $8$ 
vertices is a counterexample  to Conjecture~\ref{conjecture}, then we must have a strict inequality in~\eqref{eq:Fallat} for either $G$ or $G^c$. 
To see this, suppose that
\[d(G)+1=q(G)\geq q(G^c)=d(G^c)+1\] and \[q(G)+q(G^c)\geq |G|+3\ge 11.\] We must have $q(G)\geq 6$, so $\diam(G)\ge d(G)\ge 5$ which implies that $\diam(G^c)\le 2$ (see
e.g.,~\cite[Exercise~1.6.12]{MR0411988}). Hence \[q(G^c)=d(G^c)+1\le \diam(G^c)+1\le 3,\] which is a contradiction by Corollary~\ref{cor:q leq 4}.
\end{remark}

\begin{remark}
In this final remark we suggest a few open problems, that can be viewed as stepping stones towards a full resolution of Conjecture~\ref{conjecture}.
\begin{enumerate}
\item  Find an alternative proof of Theorem~\ref{lemma:67} that does not depend on the computer-aided case by case analysis above. This may suggest how to resolve the conjecture in some other cases. 
\item Prove that the following join inequality holds:$$q(G \vee H) \leq q(G)+q(H).$$ 
Since $q(G\cup H)\le q(G)+q(H)-2$ and $(G\cup H)^c=G^c\vee H^c$, 
this would allow us in many cases to reduce the conjecture to connected graphs $G$ and $G^c$.
\item Find matrices with the SMP that achieve $q(G)$ for families of graphs $G$ considered in this paper, such as  the complements of trees. This would enable us to use powerful tools developed in~\cite{MR3665573} to advance the conjecture. 
\end{enumerate}
\end{remark}

\bibliographystyle{amsplain}
\bibliography{qqCref}
\end{document}